\documentclass[12pt]{article}
\usepackage{amssymb}
\usepackage{mathrsfs}
\usepackage{amsmath}
\usepackage{amsfonts,amsthm,amssymb}
\usepackage{amsfonts}
\usepackage{graphics}
\usepackage{color}
\usepackage{graphicx}
\usepackage{epstopdf}
\usepackage{subfigure}
\usepackage{appendix}

\newtheorem{lem}{Lemma}[section]
\newtheorem{thm}[lem]{Theorem}

\newtheorem{Proposition}[lem]{Proposition}

\begin{document}

\title{Menger-type connectivity of line graphs of faulty hypercubes\footnotetext{This work is supported by NSFC (No. 11971406, 12171402).}}
\author{Huanshen Jia$^{a, b}$ and Jianguo Qian$^{a,}$\footnote{Corresponding author. E-mail: jgqian@xmu.edu.cn (J.G. Qian)}\\
\small $^{a}$School of Mathematical Sciences, Xiamen University, Xiamen 361005, PR China\\
\small $^{b}$School of Mathematics and Statistics, Qinghai Minzu University, Xining, 810007, PR China}
\date{}
\maketitle
{\small{\bf Abstract.}\quad
A connected graph $G$ is called strongly Menger edge connected if $G$ has min\{deg$_G(x)$, deg$_G(y)$\} edge-disjoint paths between any two distinct vertices $x$ and $y$ in $G$.  In this paper, we consider two types of strongly Menger edge connectivity of the line graphs of $n$-dimensional hypercube-like networks with faulty edges, namely the  $m$-edge-fault-tolerant and $m$-conditional edge-fault-tolerant strongly Menger edge connectivity. We show that the line graph of any $n$-dimensional hypercube-like network  is $(2n-4)$-edge-fault-tolerant strongly Menger edge connected for $n\geq 3$ and $(4n-10)$-conditional edge-fault-tolerant strongly Menger edge connected for $n\geq 4$. The two bounds for the maximum number of faulty edges are best possible.}


\section{Introduction}

With the large-scale increase in the number of nodes or links in a network, the probability of  node or link failure is very high, and fault tolerance has become an inevitable measure to ensure reliable communication. The node connectivity and link connectivity are topologically two important parameters to measure the fault tolerance and the reliability of networks. That is, the higher the connectivity is, the more reliable the network is.

In terms of graph theoretical terminology, the following Menger's Theorem provides a local point of view concerning the node connectivity and link  connectivity.
\begin{thm}\label{thm 1.1}  \cite{Menger} For any two non-adjacent vertices $u$ and $v$ in a graph $G$,
 	
(1).  the minimum size of an $(u, v)$-cut in $G$ equals the maximum number of disjoint $(u, v)$-path;
 	
(2). the minimum size of an $(u, v)$-edge cut in $G$ equals the maximum number of edge-disjoint $(u, v)$-path.
 \end{thm}

 Based on the above theorem, various extensions of Menger-type connectivity were introduced to meet practical requirements in the literature. In \cite{Oh}, Oh et al. proposed the notion of strong Menger connectivity, which is also referred as the maximal local-connectivity \cite{Chen}. Qiao et al. \cite{Qiao} and Li et al. \cite{Li4} introduced strong Menger edge-connectivity as follows.

 \noindent{\bf Definition 1.2} \cite{Qiao,Li4}.
 A connected graph $G$ is called strongly Menger edge connected if, for any two distinct vertices $u$ and $v$ in $G$, there are min$\{\deg_G(u),\deg_G(v)\}$ edge-disjoint $(u,v)$-paths between $u$ and $v$.

Further, in the same paper,  Li et al. introduced two types of  edge fault tolerance of a graph with respect to strongly Menger edge connectivity.

 \noindent{\bf Definition 1.3} \cite{Li4}.
 Let $m$ be a positive integer and $G$ a connected graph. Then

 (1). $G$ is $m$-edge-fault-tolerant strongly Menger edge connected if $G-F$ is strongly Menger edge connected for any edge set $F\subseteq E(G)$ with $|F|\leq m$.

 (2). $G$ is $m$-conditional edge-fault-tolerant strongly Menger edge connected if $G-F$ is strongly Menger edge connected for any edge set $F\subseteq E(G)$ with $|F|\leq m$ and $\delta (G-F)\geq 2$.

 The study on fault-tolerant strongly Menger edge connectivity of various types of hypercubes received wide attention. Qiao et al. \cite{Qiao} proved that all $n$-dimensional hypercubes (resp. folded hypercubes) are $(2n-4)$ (resp. $(2n-2)$-conditional edge-fault-tolerant strongly Menger edge connected for $n\geq 5$. Cheng et al. \cite{Cheng2} proved that the $n$-dimensional folded hypercube is $(3n-5)$-conditional edge-fault-tolerant strongly Menger edge connected for $n\geq 5$. Li et al. \cite{Li4} proved that the $n$-dimensional balanced hypercube is $(2n-2)$-edge-fault-tolerant strongly Menger edge connected and $(6n-8)$-conditional edge-fault-tolerant strongly Menger edge connected for $n\geq 2$. Li et al. \cite{Li5} also proved that all $n$-dimensional hypercube-like networks are $(n-2)$-edge-fault-tolerant strongly Menger edge connected and $(3n-8)$-conditional edge-fault-tolerant strongly Menger edge connected for $n\geq 3$. Ma et al. \cite{Ma} proved that the augmented cube is $(4n-8)$-conditional edge-fault-tolerant strongly Menger edge connected for $n\geq 4$.

 In recent years, a variety of data center networks have been proposed based on some special  topology, such as BCube\cite{Guo} based on generalized hypercube, CamCube\cite{Abu-Libdeh} based on $k$-ary $3$-cube and Fat-tree\cite{Al-Fares} based on fat-trees interconnection network. Among many interconnection networks, hypercube is one of the most versatile and efficient networks for parallel computation. However, when dealing with very large scale parallel computers, the use of hypercube networks is greatly limited due to technical limitations. Therefore, several variants of hypercubes are proposed based on the hypercube, such as crossed cube, M$\ddot{o}$bius cube and locally twisted cubes, which overcome the above limitation, and their diameters are about half of the diameter of the hypercubes. Researchers proposed a unified definition of these variants, Vaidya et al. \cite{Vaidya} introduced a class of hypercube-like interconnection networks, called HL-graphs, which are also called BC networks by Fan et al. \cite{Fan}. We refer to \cite{Li5, Li6, Zhou} for more results about hypercube-like networks.

 Theoretically, a hypercube-like network can be generally defined as follows. Let $G_1$ and $G_2$ be two disjoint graphs with $n$ vertices. Let $f$ be a bijection between $V(G_1)$ and $V(G_2)$, i.e., $f:V(G_1)\rightarrow V(G_2)$. We denote by $G_1\oplus_f G_2$ the graph obtained from $G_1$ and $G_2$ by adding $n$ new edges $(v,f(v)), v\in V(G_1)$ and call them the $f$-{\it edges}. In terms of the operation $\oplus_f$, an $n$-dimensional hypercube-like network is defined recursively as follows:

 (1). $K_2$ is the only $1$-dimensional hypercube-like network;

 (2). For $n\geq 2$, if $Q^1_{n-1}$ and $Q^2_{n-1}$ are two $(n-1)$-dimensional hypercube-like networks then $Q^1_{n-1}\oplus_fQ^2_{n-1}$ is an $n$-dimensional hypercube-like network for any bijection $f$.

By the above definition, it is convenient to code a vertex of an $n$-dimensional hypercube-like network by a $\{0,1\}$-string $a_na_{n-1}\cdots a_1$ of length $n$. For an example, if $f$ is the identical bijection at each recurrence step, then the resulting network is exactly the usual hypercube network, in which two vertices $0a_{n-1}\cdots a_1$ and $1b_{n-1}\cdots b_1$ are joined by an edge if and only if $a_i=b_i$ for every $i\in\{1,2,\ldots,n-1\}$. Intuitively, we can see that the $f$-edges in a usual hypercube $Q^1_{n-1}\oplus_fQ^2_{n-1}$ network  are `not crossed', as illustrated in Figure 1 (a) when $n=3$. In contrast, we can also choose a proper bijection $f$ so that the $f$-edges are `crossed'. For an example, Efe \cite{Efe} introduced a type of edge-crossed hypercube-like network, called the {\it $n$-dimensional crossed cube} and denoted by $CQ_n$, in which the bijection $f$ admits a particular rule, that is, two nodes $0a_{n-1}\cdots a_1$ and $1b_{n-1}\cdots b_1$ are joined by an edge  if and only if\\
 1). $a_{n-1}=b_{n-1}$ if $n$ is even, and\\
 2). $a_{2i}a_{2i-1} \sim b_{2i}b_{2i-1}$ for $i$ with $1\leq i<\lfloor (n-1)/2\rfloor+1$,\\
 where, for two strings $aa'$ and $bb'$ of length 2, $aa'\sim bb'$ means that $(aa',bb')\in\{(00,00),(10,10),(01,11),(11,01)\}$ and are called {\it pair-related}  \cite{Efe}. When $n=3$, the $f$-edges in the $3$-dimensional crossed cube is illustrated as in Figure 1 (b).
\begin{figure}[htp]
\begin{center}
\includegraphics[width=140mm]{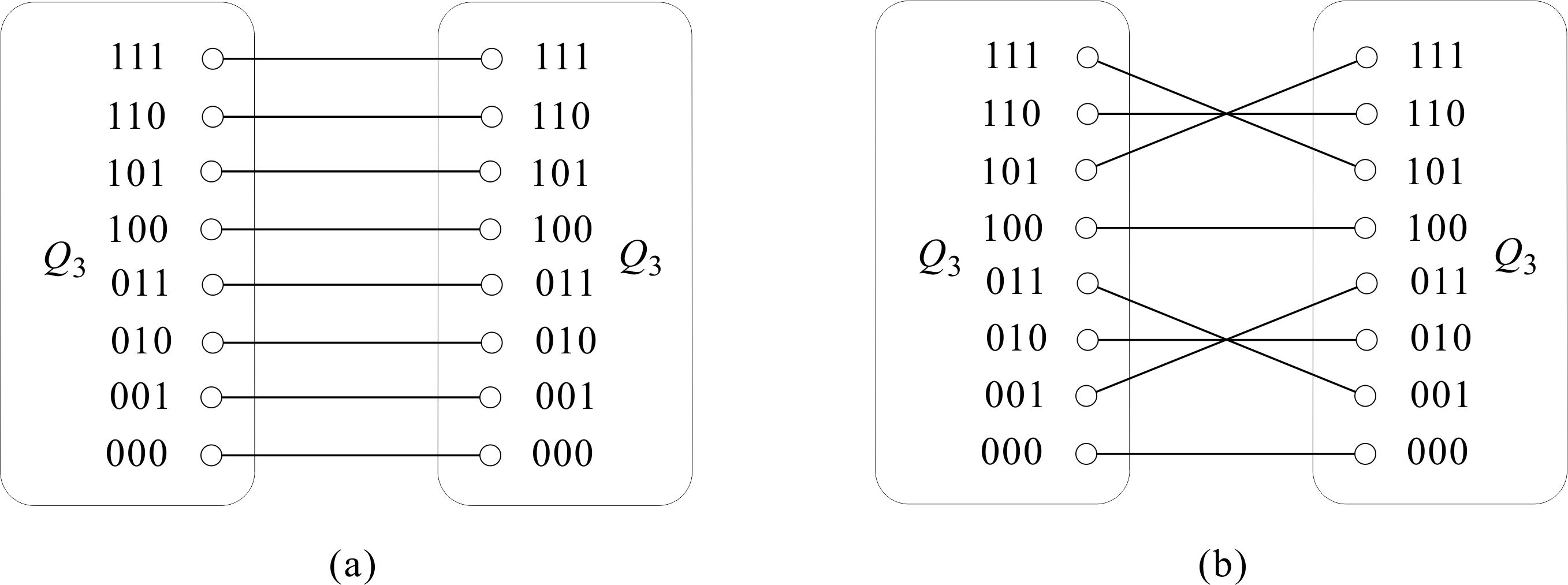}
\caption{(a). The identical bijection; (b). A crossed bijection.}
\end{center}
\end{figure}

 Indeed, based on various edge-cross patterns, a number of edge-crossed hypercube-like networks were defined in the literature to meet some particular requirements: Cull et al. \cite{Cull} introduced a kind of hypercube-like networks, called the $n$-dimensional {\it 0-M$\ddot{o}$bius cube} $0$-$Q_n$ (resp. {\it 1-M$\ddot{o}$bius cube} $1$-$Q_n$), in which the bijection $f$ admits a particular rule, that is, two vertices $0a_{n-1}\cdots a_1$ and $1b_{n-1}\cdots b_1$ are joined by an edge if and only if $a_i=b_i$ (resp. $a_i=\bar {b}_i$) for every $i\in\{1,2,\ldots,n-1\}$, where for a binary bit $b\in \{0,1\}$, $\bar{b}=1$ if and only if $b=0$. Yang et al. \cite{Yang} proposed the locally twisted cube $ LTQ_n $, in which two vertices $0a_{n-1}\cdots a_1$ and $1b_{n-1}\cdots b_1$ are joined by an edge if and only if $0a_{n-1}a_{n-2} \cdots a_1=1(b_{n-1}+b_1)b_{n-2}\cdots b_1$, where, for $a,b \in \{0,1\}^n$, $a+b$ denote the (bitwise modulo $2$) sum of $a$ and $b$.

In this paper, we are interested in a variety of  hypercube-like networks, namely the line graph of   hypercube-like networks.  The line graph $L(G)$ of a graph $G$ is the graph with $|V (L(G))|=|E(G)|$, such that two vertices of $L(G)$ are adjacent provided their corresponding edges share (incident with) a common end-vertex of $G$. Based on its original graph, the line graph may inherit some nice structural properties. For an example, the line graph of a graph with higher connectivity or smaller diameter may have  higher connectivity or smaller diameter either  \cite{Chartrand,Zhang}.

 In practice, the line graph of hypercube-like networks may serve as a nice topology for Data center networks. With the development of data center and cloud computing, data center network has become a platform to connect large-scale servers of data center to realize online cloud services. The performance of cloud computing largely depends on the performance of the data center network. In 2018, Wang et al. \cite{Wang} proposed a new high-performance and server-centric data center network BCDC. The $n$-dimensional  BCDC is defined on a particular hypercube-like network, that is, the $n$-dimensional crossed cube $CQ_n$ (see above for its definition), by inserting a new node at each edge of  $CQ_n$. The resulting graph, denoted by $A_n$, is called the {\it original graph} of the  BCDC, in which the nodes in  $CQ_n$ are the switches while the inserted nodes are the servers. The {\it logical graph} of the BCDC is therefore defined as the line graph of   $CQ_n$ and is denoted by $B_n$.  Figure 2 illustrates the original graph and logical graph of the $3$-dimensional  BCDC in which the server nodes are coded by the pair of the codes of its two adjacent switch nodes.
\begin{figure}[htbp]
	\centering
	{
		\begin{minipage}[t]{1\textwidth}
			\centering
			\includegraphics[width=13cm]{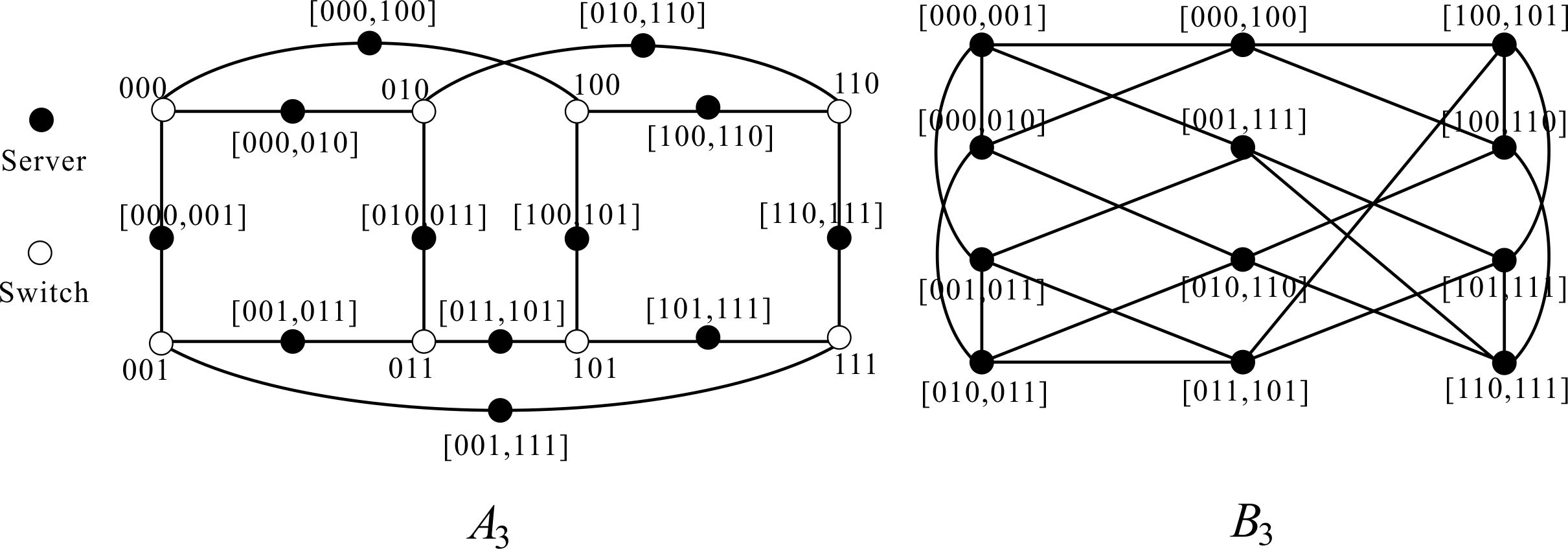}
			\caption{The original graph $A_3$ and logical graph $B_3$ of the $3$-dimensional BCDC.}
		\end{minipage}%
	}%
\end{figure}

As pointed in \cite{Wang}, BCDC meets the goal of using only low-end
commodity switches by putting routing computation into servers purely. Kan et al. \cite{Kan} showed that $B_n$ is superior to other data center network such as DCell \cite{Lv1} and Fat-Tree \cite{Al-Fares} through practical experiment and analyzed the data communication, fault tolerance and node-disjoint paths. Recently, Lv et al. \cite{Lv2} studied the $g$-restricted connectivity and $h$-extra connectivity of $B_n$.

Motivated by BCDC networks, we consider the connectivity of the  line graphs of general $n$-dimensional hypercube-like networks. We prove that  the line graph of any $n$-dimensional hypercube-like network is $(2n-4)$-edge-fault-tolerant strongly Menger edge connected for $n\geq 3$ and $(4n-10)$-conditional edge-fault-tolerant strongly Menger edge connected for $n\geq 4$. The two bounds for the maximum number of faulty edges are best possible.

\section{Preliminaries}

In the following we use graph theoretical terminologies. For a graph $G$, we use $V(G)$ and $E(G)$ to denote the vertex set and edge set of $G$, respectively.  For a vertex $u\in V(G)$, we use $N_G(u)$ to denote the set of the vertices adjacent to $u$ in $G$ and use $\deg_G(u)$ to denote the degree of $u$ in $G$, i.e.,  $\deg_G(u)=|N_G(u)|$. In particular, if $\deg_G(u)=k$ for every $u\in V(G)$, then we call $G$ $k$-{\it regular}. The  minimum degree of a graph $G$ is defined by $\delta (G)=\min\{\deg_G(v): v\in V(G)\}$. For a set $F\subseteq V(G)\cup E(G)$, we use $G-F$ to denote the graph obtained from $G$ by deleting $F$ from $G$. For two disjoint vertex sets or subgraphs $H_1$ and $H_2$, we use $E(H_1,H_2)$ to denote the set of the edges with one end in $H_1$ and the other in $H_2$. If $H_1=\{u\}$, we use $E(u, H_2)$ instead of $E(\{u\}, H_2)$ for simplicity. For two vertices $v_0, v_k\in V(G)$, an $(v_0, v_k)$-{\it path} of length $k$ is a finite sequence of distinct vertices $v_0v_1\cdots v_{k}$ such that $(v_i,v_{i+1})\in E(G)$ for $0\leq i\leq k-1$. A set $F\subseteq E(G)$ is an $(u, v)$-{\it edge cut} if $G-F$ has no $(u, v)$-path. The connectivity $\kappa (G)$ of $G$ is the minimum size of a vertex set $F$ such that $G-F$ is disconnected or has only one vertex. The edge connectivity $\lambda(G)$ of $G$ is defined analogously.

In the following, we use $Q_n$ to denote an arbitrary $n$-dimensional hypercube-like network. For $n\geq 2$, we denote by ${\cal L}_n$ the class of the line graphs of all $n$-dimensional hypercube-like networks $HL_n $, i.e.,
$${\cal L}_n=\{L(Q_n): Q_n\in HL_n\}.$$

For an $n$-dimensional hypercube-like network $Q_n=Q^1_{n-1}\oplus_fQ^2_{n-1}$, we call a vertex of the line graph $L(Q_n)$ that corresponds to an $f$-edge of $Q_n$ an $f$-{\it vertex} of $L(Q_n)$. For $L_n=L(Q^1_{n-1}\oplus_fQ^2_{n-1})$, we denote  $L_{n-1}^1=L(Q_{n-1}^1),L_{n-1}^2=L(Q_{n-1}^2)$ and denote by $F_n$ the set of the $f$-vertices of $L_n$.
\begin{lem}\label{lem 2.1.} \cite{Bondy}
For any graph $G$, $\kappa(G)\leq \lambda (G) \leq \delta(G)$.
\end{lem}
\begin{lem}\label{lem 2.2.} \cite{Fan}
	For $Q_n \in HL_n$, $\kappa(Q_n)=\lambda(Q_n)= n$ for any $n \geq 1$.	
\end{lem}
\begin{thm}\label{thm 2.3.} \cite{Chartrand}
	If a graph $G$ is $n$-edge-connected, then its line graph $L(G)$ is $(2n-2)$-edge-connected.
\end{thm}

\begin{lem}\label{lem 2.4.}
	Let $L_n\in{\cal L}_n$. Then $\kappa(L_n)=\lambda (L_n)=2n-2$ for $n \geq 2$.
\end{lem}
\begin{proof}
	By Lemma \ref{lem 2.2.} and Theorem \ref{thm 2.3.}, we have $L_n$ is $(2n-2)$-edge-connected. Since $L_n$ is a $(2n-2)$-regular graph and by Lemma \ref{lem 2.1.}, we have $\kappa(L_n)=\lambda (L_n)=2n-2$.
\end{proof}

\section{Edge fault-tolerant strong Menger edge connectivity of  line graph of hypercube-like networks}

The following proposition follows directly from the  definition of $L(Q_n)$.
\begin{Proposition}\label{prop}
Every vertex in $L_{n-1}^1\cup L_{n-1}^2$ is adjacent to two vertices in $F_n$. Conversely, every vertex in $F_n$ is adjacent to $n-1$ vertices in $L_{n-1}^1$ and $n-1$ vertices in $L_{n-1}^2$.
\end{Proposition}
\begin{lem}\label{lem 3.2.}
For $n\geq 3, L_n\in{\cal L}_n$ and $S\subset E(L_n)$, if $|S|\leq 4n-7$, then  $L_n-S$ has a connected component with at least $n2^{n-1}-1$ vertices.
\end{lem}
\begin{proof}
Assume $L_n=L(Q^1_{n-1}\oplus_fQ^2_{n-1})$. Let $E_1=E(L_{n-1}^1),E_2=E(L_{n-1}^2),\ E_f=E(F_n,L_{n-1}^1\cup L_{n-1}^2)$ and $S_1=S\cap E_1,S_2=S\cap E_2,S_f=S\cap E_f$. We can see that $E(L_n)=E_1\cup E_2\cup E_f$ and $S=S_1\cup S_2\cup S_f$ are partitions of $E(L_n)$ and $S$, respectively. Without loss of generality, we assume $|S_1|\leq |S_2|$ and, hence, $|S_1|\leq \lfloor|S|/2\rfloor=2n-4$.

We prove the lemma by induction on $n$. When $n=3$, the assertion follows by a direct observation. We now assume that the lemma holds for $n-1$ when $n\geq 4$.
	
{\bf Case 1.} $L_{n-1}^1-S_1$ is disconnected and $L_{n-1}^2-S_2$ is connected.
	
Since $n\geq 4,|S|\leq 4n-7$, we have $|S_1|\leq 2n-4\leq 4(n-1)-7$. So by the induction hypothesis, $L_{n-1}^1-S_1$ has a connected component $H_1$ with at least $(n-1)2^{n-2}-1$ vertices. On the other hand, by Lemma \ref{lem 2.4.}, $\lambda(L_{n-1}^1)=2(n-1)-2=2n-4$. Therefore, $|S_1|\geq 2n-4$, meaning that $|S_1|=2n-4$. Further, since $|S_2|\geq |S_1|$, we have $|S_2|\geq 2n-4$ and, thus,
	\begin{equation}\label{1}
		|S_f|=|S|-|S_1|-|S_2|\leq 4n-7-2(2n-4)=1.
	\end{equation}
We notice that $L_{n-1}^1-H_1$ has only one vertex. So by Proposition \ref{prop}, for any vertex $v$ in $L_{n-1}^2$,  $L_n$ has two disjoint paths $P_1=vw_1u_1$ and $P_2=vw_2u_2$, where $w_1,w_2\in F_n,w_1\not=w_2$ and $u_1,u_2\in V(H_1)$. Therefore, (\ref{1}) means that  at least one of $P_1$ and $P_2$ is a path in $L_n-S$ and, hence, $v$ is connected to a vertex in $H_1$. Consequently, $L_n-S-v_0$ is connected, where $v_0$ is the only vertex in $L_{n-1}^1-H_1$.
	
{\bf Case 2.} $L_{n-1}^1-S_1$ and $L_{n-1}^2-S_2$ are both connected.
	
Recall that every vertex of $L_n$ has degree $2(n-1)$. On the other hand, since $|S|\leq 4n-7<2\times 2(n-1)$, $L_n-S$ has at most one isolated vertex (a vertex not adjacent to any other vertex). This means that, except the possible isolated vertex $v_0$, every vertex in $F_n$ is not an isolated vertex and, hence, adjacent to either a vertex in $L^1_{n-1}$ or a vertex in $L_{n-1}^2$. Moreover, at least one vertex in $F_n\setminus\{v_0\}$ is adjacent to both a vertex in $L^1_{n-1}$ and a vertex in $L_{n-1}^2$. For otherwise, every vertex in $F_n$ would be incident with at least $n-1$ edges in $S$ by Proposition \ref{prop} and, hence, $|S|\geq (n-1)|F_n-1|>4n-7$ as $n\geq 4$ and $|F_n|=2^{n-1}\geq 8$, a contradiction. This means that $L_n-S-v_0$ is connected, which has exactly $n2^{n-1}-1$ vertices.
	
{\bf Case 3.} $L_{n-1}^1-S_1$ is connected and $L_{n-1}^2-S_2$ is disconnected.
	
Since $\lambda(L_{n-1}^2)=2n-4$, we have $|S_2|\geq2n-4$.
	
{\bf Case 3.1.} $2n-4\leq|S_2|\leq4(n-1)-7$.
	
By the induction hypothesis, $L_{n-1}^2-S_2$ has a connected component $H_2$ with at least $(n-1)2^{n-2}-1$ vertices. Further, since $L_{n-1}^2-S_2$ is disconnected, $L_{n-1}^2-S_2$ has exactly two components, one of which is $H_2$ with $(n-1)2^{n-2}-1$ vertices and the other is an isolated vertex, say $v_0$.
	
Recalling that $|S_2|\geq 2n-4$, it follows that $|S_f|\leq|S|-|S_2|\leq4n-7-(2n-4)=2n-3<2(n-1)$. So by Proposition \ref{prop}, except possible one vertex, say $u_0\in F_n$, every vertex in $F_n$ other than $u_0$ is adjacent to both a vertex in $L_{n-1}^1$ and a vertex in $L_{n-1}^2$. Further, at least one vertex in $F_n\setminus\{u_0\}$ is adjacent to a vertex in $H_2$. For otherwise, we would have $|S|\geq (n-2)(|F_n|-1)>4n-7$ because $n\geq 4,|F_n|=2^{n-1}\geq 8$ and every vertex in $F_n$ is adjacent to $n-1$ vertices in $L_{n-1}^2$ and hence adjacent to at least $n-2$ vertices in $H_2$. This is a contradiction. As a result, the subgraph of $L_n-S$ induced by $V(L_{n-1}^1)\cup V(H_2)\cup (F_n\setminus\{u_0\})$ is connected. Therefore, if $v_0$ is adjacent to a vertex in $F_n\setminus\{u_0\}$, then we obtain a connected component with vertex set $V(L_n)\setminus\{u_0\}$.

We now assume that $v_0$ is not adjacent to any vertex in $F_n\setminus\{u_0\}$. In this case, $v_0$ is adjacent only to $u_0$ or is an isolated vertex in $L_n-S$. This implies that the number of edges in $S$ incident with $v_0$ is at least $2(n-1)-1=2n-3$ and, conversely, $S$ has at most $4n-7-(2n-3)=2n-4$ edges that are not incident with $v_0$. Notice that $u_0$ has degree $2(n-1)$ in $L_n$, meaning that $u_0$ has degree at least $2(n-1)-(2n-4)=2$ in $L_n-S$. Therefore, $u_0$ is adjacent to at least one vertex in $L_{n-1}^1\cup H_2$. Hence, the subgraph induced by $V(L_n)\setminus\{v_0\}$ is connected, which has $n2^{n-1}-1$ vertices, as desired.

{\bf Case 3.2.} $4n-10\leq|S_2|\leq4n-7$.
	
In this case we have $|S_f|\leq|S|-|S_2|\leq3$. So  by Proposition \ref{prop}, except the particular case when $n=4$ and $S_f=(u_0,L_{n-1}^1)$ for some $u_0\in F_n$, every vertex in $F_n$ is adjacent to both a vertex in $L_{n-1}^1$ and a vertex in $L_{n-1}^2$. Therefore, $L_n-S$ is connected as $L_{n-1}^1$ is connected. For the particular case, it is clear that every vertex in $F_n\setminus\{u_0\}$ is adjacent to both a vertex in $L_{n-1}^1$ and a vertex in $L_{n-1}^2$, meaning that $L_n-S-u_0$ is connected.  \end{proof}

\begin{thm}\label{thm 3.3}
For any $L_n\in{\cal L}_n$, $L_n$ is $(2n-4)$-edge-fault-tolerant strongly Menger edge connected. Further, if $k>2n-4$, then $L_n$ is not $k$-edge-fault-tolerant strongly Menger edge connected.
\end{thm}

\begin{proof} If $n$ is less than 3, then the theorem trivially holds. We now assume that $n\geq 3$. Let $T\subset E(L_n)$ with $|T|\leq 2n-4$. It suffices to prove that $L_n-T$ has $\min\{\deg_{L_n-T}(u), \deg_{L_n-T}(v)\}$ edge-disjoint paths connecting any two vertices $u,v\in V(L_n)$.

Since $\lambda(L_n)=2n-2$, $L_n-T$ is connected. Let $u$ and $v$ be any two different vertices of $L_n$. Without loss of generality, we assume that deg$_{L_n-T}(u)=\min\{{\rm deg}_{L_n-T}(u), {\rm deg}_{L_n-T}(v)\}$. By Theorem \ref{thm 1.1}, we need to show that the minimum size of a $(u,v)$-edge cut is deg$_{L_n-T}(u)$. Suppose, to the contrary, that $u$ and $v$ are disconnected by deleting a set $F$ of edges with $|F|\leq {\rm deg}_{L_n-T}(u)-1$. Since deg$_{L_n-T}(u)\leq {\rm deg}_{L_n}(u)=2n-2$, we have $|F|\leq 2n-3$.
	
Let $S=T\cup F$. Then $|S|\leq 4n-7$. So by Lemma \ref{lem 3.2.}, $L_n-S$ has a connected component $H$ such that $|V(H)|\geq n2^{n-1}-1$. Since $u$ is disconnected from $v$, we have $|V(H)|= n2^{n-1}-1$ and $|V(L_n)\setminus V(H)|=1$. Therefore, $u\in V(H)$ and $V(L_n)\setminus V(H)=\{v\}$, or $v\in V(H)$ and $V(L_n)\setminus V(H)=\{u\}$. If  $v\in V(H)$ and $V(L_n)\setminus V(H)=\{u\}$, then $E(u,N_{L_n-T}(u))\subseteq F$. Hence, $|F|\geq \deg_{L_n-T}(u)$, which contradicts $|F|\leq \deg_{L_n-T}(u)-1$. Similarly, if  $u\in V(H)$ and $V(L_n)\setminus V(H)=\{v\}$, then $E(v,N_{L_n-T}(v))\subseteq F$. Hence, $|F|\geq \deg_{L_n-T}(v)$, which is again a contradiction to $|F|\leq \deg_{L_n-T}(u)-1$.
	
Therefore, $L_n$ is $(2n-4)$-edge-fault-tolerant strongly Menger edge connected.

Finally, choose arbitrary two adjacent vertices $u_0$ and $u$ in $V(L_n)$. Let $v$ be a vertex in $V(L_n)\setminus(N_{L_n}(u_0)\cup\{u_0\})$ and let $S=E(u_0, N_{L_n}(u_0)\setminus\{u\})$, see Figure 3. Hence $|S|=2n-3$. By a direct observation, $L_n-S$ has no more than $(2n-3)$ edge-disjoint paths connecting $u$ and $v$. Notice that $\deg_{L_n}(u)=\deg_{L_n}(v)=2n-2$. This means that $L_n$ is not $(2n-3)$-edge-fault-tolerant strongly Menger edge connected. \end{proof}

\begin{figure}[htbp]
	\centering
	{
		\begin{minipage}[t]{0.9\textwidth}
			\centering
			\includegraphics[width=10cm]{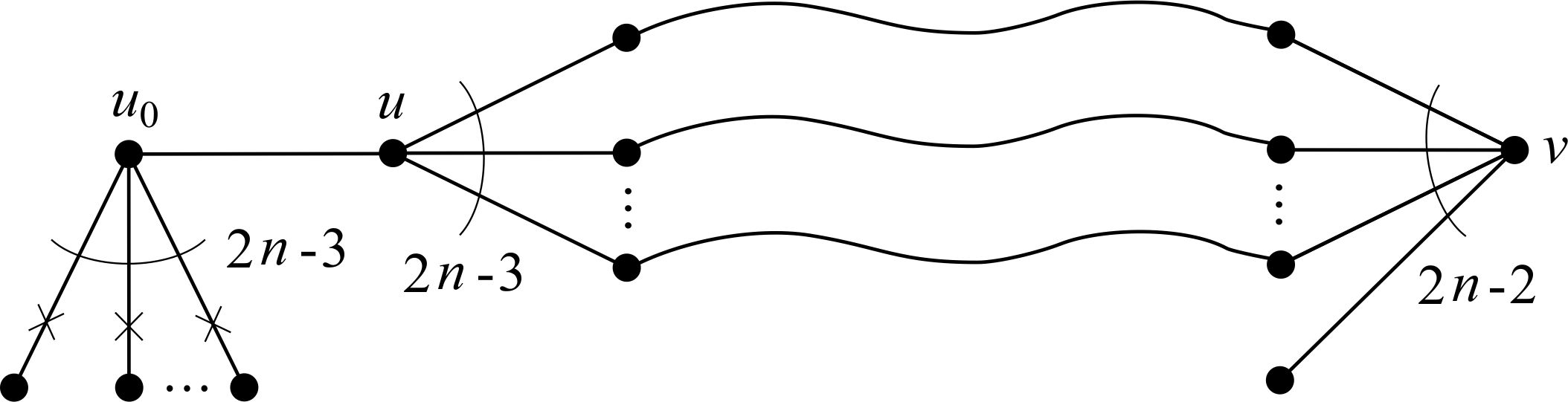}
			\caption{$L_n-S$ has no more than $(2n-3)$ edge-disjoint paths connecting $u$ and $v$.}
		\end{minipage}%
	}%
\end{figure}

\section{Conditional edge fault-tolerant strong Menger edge connectivity of line graph of hypercube-like networks}

\begin{lem}\label{lem 4.1.}
For $n\geq 4$, $L_n\in{\cal L}_n$ and $S\subset E(L_n)$, if $|S|\leq 6n-13$, then $L_n-S$ has a connected component with at least $n2^{n-1}-2$ vertices.
\end{lem}

\begin{proof} Assume $L_n=L(Q^1_{n-1}\oplus_fQ^2_{n-1})$. Let $E_1=E(L_{n-1}^1),E_2=E(L_{n-1}^2),E_f=E(F_n,L_{n-1}^1\cup L_{n-1}^2)$ and $S_1=S\cap E_1,S_2=S\cap E_2,S_f=S\cap E_f$. It is clear that $E(L_n)=E_1\cup E_2\cup E_f$ and $S=S_1\cup S_2\cup S_f$ are partitions of $E(L_n)$ and $S$, respectively. Without loss of generality, we assume $|S_1|\leq |S_2|$ and, hence, $|S_1| \leq \lfloor|S|/2\rfloor=3n-7$.

 We prove the lemma by induction on $n$. When $n=4$, the conclusion is true and the proof is shown in Appendix A. Assume now $n\geq 5$.
	
{\bf Case 1.} $L_{n-1}^1-S_1$ and $L_{n-1}^2-S_2$ are both connected.

Since $|S|\leq 6n-13<3\times2(n-1)$ and every vertex of $L_n$ has degree $2(n-1)$, $L_n-S$ has at most two isolated vertices. This means that, except two possible isolated vertices, say $u_0,u_1\in F_n$, every vertex in $F_n$ is adjacent to either a vertex in $L_{n-1}^1$ or a vertex in $L_{n-1}^2$. Further, at least one vertex in $F_n\setminus \{u_0, u_1\}$ is adjacent to both a vertex in $L_{n-1}^1$ and a vertex in $L_{n-1}^2$. For otherwise, every vertex in $F_n$ would be incident with at least $n-1$ edges in $S$  and hence, by Proposition \ref{prop}, $|S|\geq (n-1)(|F_n|-2)>6n-13$ as $n\geq 5$ and $|F_n|=2^{n-1}\geq 16$, a contradiction. Therefore, the subgraph induced by $V(L_n)\setminus \{u_0,u_1\}$ is connected, which has $n2^{n-1}-2$ vertices.

{\bf Case 2.} $L_{n-1}^1-S_1$ is disconnected and $L_{n-1}^2-S_2$ is connected.
	
Since $n\geq5$ and $|S_1|\leq 3n-7<6(n-1)-13$, $L_{n-1}^1-S_1$ has a connected component $H_1$ with at least $(n-1)2^{n-2}-2$ vertices by the induction hypothesis. On the other hand, by Lemma \ref{lem 2.4.}, $ \lambda(L_{n-1}^1)= 2n-4$, and hence $ 2n-4\leq |S_1|\leq 3n-7<2(2n-4)-2=4n-10$. This means that $ L_{n-1}^1-H_1$ has only one vertex and $H_1$ has exactly $(n-1)2^{n-2}-1$ vertices. Further, since $|S_2|\geq |S_1|$, we have $|S_2|\geq 2n-4$ and thus, $|S_f|\leq |S|-|S_1|-|S_2|\leq6n-13-2(2n-4)=2n-5<2(n-1)$. So by Proposition \ref{prop}, except possible one vertex, say $ u_0\in F_n $, every vertex in $F_n$ other than $u_0$ is adjacent to both a vertex in $ L_{n-1}^1 $ and a vertex in $ L_{n-1}^2 $. Further, at least one vertex in $F_n\setminus \{u_0\}$ is adjacent to a vertex in $H_1$. For otherwise, we would have $|S|\geq (n-2)(|F_n|-1)>6n-13$ as $n\geq 5$, $|F_n|=2^{n-1}\geq 16$ and every vertex in $F_n$ is adjacent to $n-1$ vertices in $L_{n-1}^1$ and hence adjacent to at least $n-2$ vertices in $H_1$. This is a contradiction. Therefore, the subgraph induced by $V(H_1)\cup V(L_{n-1}^2)\cup (F_n\setminus \{u_0\})$ is connected, which has $n2^{n-1}-2$ vertices.

{\bf Case 3.} $L_{n-1}^1-S_1$ is connected and $L_{n-1}^2-S_2$ is disconnected.
	
Since $ \lambda (L_{n-1}^2)=2n-4 $, we have $ |S_2|\geq 2n-4 $.
	
{\bf Case 3.1.} $2n-4\leq |S_2|\leq 6(n-1)-13=6n-19$.
	
By the induction hypothesis, $L_{n-1}^2-S_2$ has a connected component $H_2$ with at least $(n-1)2^{n-2}-2$ vertices and, hence
$V(L_{n-1}^2- H_2)$ has at most two vertices, say $v_0$ and $v_1$. Recalling that $|S_2|\geq 2n-4$, it follows that $|S_f|\leq |S|-|S_2|\leq 6n-13-(2n-4)=4n-9<2\times 2(n-1)$. This implies that, except possibly two vertices $u_0,u_1\in F_n$, every vertex in $F_n$ is adjacent to both a vertex in $L_{n-1}^1$ and a vertex in $L_{n-1}^2$. Further, at least one vertex in $F_n\setminus\{u_0,u_1\}$ is adjacent to a vertex in $H_2$. For otherwise, we would have $|S|\geq (n-3)(|F_n|-2)>6n-13$ as $n\geq 5$, $|F_n|=2^{n-1}\geq 16$ and every vertex in $F_n$ is adjacent to $n-1$ vertices in $L_{n-1}^2$ and hence adjacent to at least $n-3$ vertices in $H_2$. This is a contradiction. As a result, the subgraph of $L_n-S$ induced by $V(L_{n-1}^1)\cup V(H_2)\cup (F_n\setminus \{u_0,u_1\})$ is connected.

%{\bf Case 3.1.1.} Each of $v_0$ and $v_1$ is adjacent to a vertex in $F_n\setminus \{u_0,u_1\}$.

%In this case, we can obtain a connected component with vertex set $V(L_n)\setminus \{u_0,u_1\}$, which has $n2^{n-1}-2$ vertices, as desired.

{\bf Case 3.1.1.} Either $v_0$ or $v_1$, say $v_0$, is adjacent to a vertex in $F_n\setminus \{u_0,u_1\}$.

 In this case, we obtain a connected component with vertex set $V(L_{n-1}^1)\cup V(H_2+v_0)\cup (F_n\setminus \{u_0,u_1\})$. Recall that $u_0$ and $u_1$ have degree $2(n-1)$ in $L_n$ and $|S_f|\leq 4n-9<2\times 2(n-1)=4n-4$. So by Proposition \ref{prop}, at least one of $u_0$ and $u_1$, say $u_0$, is adjacent to a vertex in $L_{n-1}^1 \cup H_2$. Therefore, we obtain a connected component with vertex set $V(L_{n-1}^1)\cup V(H_2+v_0)\cup (F_n\setminus \{u_1\})=V(L_n)\setminus \{u_1,v_1\}$, which has $n2^{n-1}-2$ vertices.

{\bf Case 3.1.2.} Neither $v_0$ nor $v_1$ is adjacent to any vertex in $F_n\setminus \{u_0,u_1\}$.

In this case, the possible vertices adjacent to $v_0$ are $u_0,u_1$ and $v_1$, meaning that $v_0$ has degree at most 3 in $L_n-S$. Similarly,  $v_1$ has degree at most 3 in $L_n-S$. On the other hand, the total degree of $v_0$ and $v_1$ in $L_n$ is $2\times 2(n-1)$. Therefore,  the number of edges in $S$ incident with $v_0$ and $v_1$ is at least $2\times 2(n-1)-2-3=4n-9$ and, conversely, $S$ has at most $6n-13-(4n-9)=2n-4$ edges that are not incident with $v_0$ and $v_1$. Notices that $u_0$ and $u_1$ has degree $2\times 2(n-1)$ in $L_n$, meaning that $u_0$ and $u_1$ has degree at least $2\times 2(n-1)-(2n-4)=2n$ in $L_n-S$. By Proposition \ref{prop}, $u_0$ and $u_1$ are adjacent to vertices in $L_{n-1}^1\cup H_2$. Therefore, the subgraph induced by $V(L_{n-1}^1)\cup V(H_2)\cup F_n$ is connected which has $n2^{n-1}-2$ vertices.

{\bf Case 3.2.} $6n-18\leq|S_2|\leq 6n-13$
	
In this case, we have $ |S_f|\leq |S|-|S_2|\leq 5 $, so by Proposition \ref{prop}, except possible one vertex, say $u_0\in F_n$, every vertex in $F_n$ is adjacent to both a vertex in $L_{n-1}^1$ and a vertex in $L_{n-1}^2$. On the other hand, $u_0$ has degree $2(n-1)$ in $L_n$ and hence, $u_0$ is adjacent to at least one vertex in $L_{n-1}^1\cup H_2$ when $n\geq 5$. Therefore, we can obtain a connected component with vertex set $V(L_n)\setminus \{v_0,v_1\}$.
	
{\bf Case 4.} $L_{n-1}^1-S_1$ and $L_{n-1}^2-S_2$ are both disconnected.
	
Since $2n-4 \leq |S_1| \leq |S_2|$ and $|S_1|\leq 3n-7<2(2n-4)-2=4n-10$, so by Lemma \ref{lem 3.2.}, $L_{n-1}^1$ has a connected component $ H_1$ with $(n-1)2^{n-2}-1$ vertices. Notice that $|S_2| \leq |S| - |S_1| \leq 6n-13-(2n-4)=4n-9<6(n-1)-13$. By the induction hypothesis, $L_{n-1}^2-S_2$ has a connected component $ H_2 $ with at least $(n-1)2^{n-2}-2$ vertices. Thus, $|S_f|=|S|-|S_1|-|S_2|\leq 6n-13-2(2n-4)=2n-5<2(n-1)$. So by Proposition \ref{prop}, except possible one vertex, say $u_0\in F_n$, every vertex in $F_n$ is adjacent to both a vertex in $L_{n-1}^1$ and a vertex in $L_{n-1}^2$. We notice that $L_{n-1}^1-H_1$ has only one vertex $v_0$, and that $L_{n-1}^2-H_2$ has two vertices $w_0, w_1$.

We now assume that none of $v_0, w_0$ and $w_1$ is adjacent to any vertex in $F_n\setminus\{u_0\}$. In this case, $v_0,w_0$ and $w_1$ are adjacent only  to $u_0$ or are isolated vertices in $L_n-S$, respectively. This implies that the number of edges in $S$ incident with $v_0,w_0$ and $w_1$ is at least $3\times 2(n-1)-4=6n-10> 6n-13$. This is a contradiction. Therefore, at most two vertices in $\{v_0, w_0, w_1\}$, say $v_0$ and $w_0$, are not adjacent to vertices in $H_1 \cup H_2$. On the other hand, since $|S_f| \leq 2n-5<2(n-1)$, $u_0$ is adjacent to at least one vertices in $H_1 \cup H_2$  by Proposition \ref{prop}. Hence, the subgraph induced by $V(L_n)\setminus \{v_0,w_0\}$ is connected which has $n2^{n-1}-2$ vertices.
\end{proof}

\begin{thm}\label{thm 3.2.}
For any $L_n\in{\cal L}_n$, $L_n$ is $(4n-10)$-conditional edge-fault-tolerant strongly Menger edge connected for any $L_n\in{\cal L}_n$. Further, if $k>4n-10$, then $L_n$ is not $k$-conditional edge-fault-tolerant strongly Menger edge connected.
\end{thm}	

\begin{proof}
If $n$ is less than $4$, then the theorem trivially holds. We now assume that $n\geq 4$. It is suffices to prove that $L_n-T$ has min$\{{\rm deg}_{L_n-T}(u),{\rm deg}_{L_n-T}(v)\}$ edge-disjoint paths connecting any two vertices $u,v\in V(L_n)$.

Let $T\subset E(L_n)$ with $|T|\leq 4n-10$ and $\delta(L_n-T)\geq 2$. Since  $|T|\leq 4n-10<4n-7$, so by Lemma \ref{lem 3.2.}, if $L_n-T$ is disconnected then  $L_n-T$ contains an isolated vertex, contradicting the assumption $\delta(L_n-T)\geq 2$. Hence, $L_n-T$ is connected. Let $u$ and $v$ be any two different vertices of $L_n$. Without loss of generality, we assume that ${\rm deg}_{L_n-T}(u)=\min\{{\rm deg}_{L_n-T}(u), {\rm deg}_{L_n-T}(v)\}$. By Theorem \ref{thm 1.1}, we need to show that the minimum size of a $(u,v)$-edge cut is $\deg_{L_n-T}(u)$. Suppose, to the contrary, that $u$ and $v$ are disconnected by deleting a set $F$ of edges with $|F|\leq {\rm deg}_{L_n-T}(u)-1$. Since ${\rm deg}_{L_n-T}(u)\leq {\rm deg}_{L_n}(u)=2n-2$. We have $|F|\leq 2n-3$.
	
Let $S=T\cup F$. Then $|S|\leq 6n-13$. By Lemma \ref{lem 4.1.}, $L_n-S$ has a connected component $H$ such that $|V(H)|\geq n2^{n-1}-2$. Since $u$ is disconnected to $v$. Without loss of generality, we assume that $u\in V(L_n)\setminus V(H)$. We consider two cases.
	
{\bf Case 1.} $|V(H)|=n2^{n-1}-1$.

In this case, $u$ is an isolated vertex in $L_n-S$. Hence, $|F|\geq |E(u,V(H))|={\rm deg}_{L_n-T}(u)$. This contradicts to $|F|\leq {\rm deg}_{L_n-T}(u)-1$.
	
{\bf Case 2.} $|V(H)|=n2^{n-1}-2$.
	
Let $w$ be a vertex in $V(L_n)\setminus V(H)$ other than $u$. Since $\delta(L_n-T)\geq 2$, we have $|E(w,V(H)) \cap F|\geq 1$. Therefore, $|F|\geq |E(\{u,w\},V(H))|\geq {\rm deg}_{L_n-T}(u)-1+1={\rm deg}_{L_n-T}(u)$. This is a contradiction.
	
Therefore, $L_n$ is $(4n-10)$-conditional edge-fault-tolerant strongly Menger edge connected.
	
Finally, choose arbitrary three adjacent vertices $u,u_1,u_2$ in $L_n$ and a vertex $u_3$ in $N_{L_n}(u_2)\setminus\{u,u_1\}$. Let $S=E[u_2, \{V(L_n)\setminus \{u,u_1,u_3\}\}]\cup E[u_1, \{V(L_n)\setminus \{u,u_2\}\}]$. Hence $|S|=4n-9$. Let $v$ be a vertex in $V(L_n)\setminus N_{L_n}(\{u,u_1,u_2\})$. By observation (see Figure 4), $L_n-S$ has no more than $(2n-3)$-edge-disjoint paths connecting $u$ and $v$. Notice that $\deg_{L_n}(u)=\deg_{L_n}(v)=2n-2$. This means that $L_n$ is not $(4n-9)$-conditional edge-fault-tolerant strongly Menger edge connected.	
\end{proof}

\begin{figure}[htbp]
	\centering
	{
		\begin{minipage}[t]{0.9\textwidth}
			\centering
			\includegraphics[width=8cm]{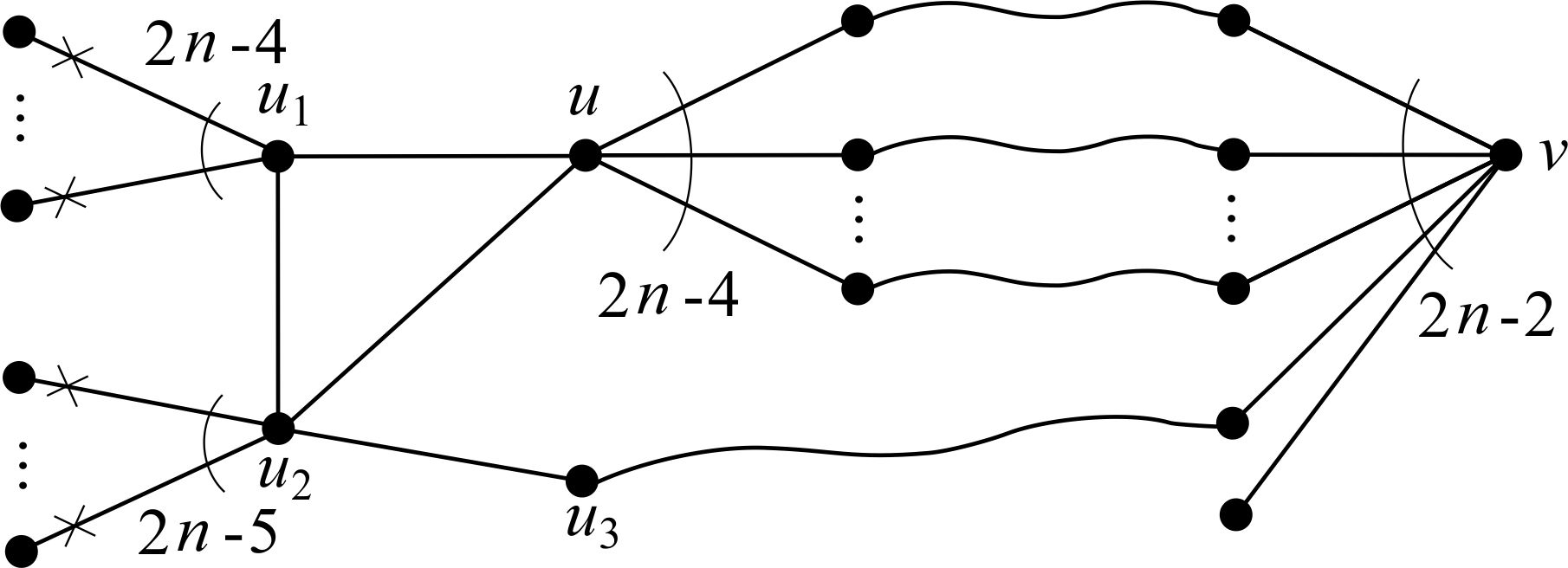}
			\caption{$L_n-S$ has no more than $(2n-3)$-edge-disjoint paths connecting $u$ and $v$.}
		\end{minipage}%
	}%
\end{figure}

\section{Conclusions}

In this paper, we study the edge-fault-tolerant strong Menger edge connectivity of the line graphs of $n$-dimensional hypercube-like networks. By exploring and utilizing the structural properties of $n$-dimensional hypercube-like network and their line graph, we show that the line graphs of $n$-dimensional hypercube-like networks is $(2n-4)$ edge-fault-tolerant strongly Menger edge connected for $n\geq 3$ and $(4n-10)$-conditional edge-fault-tolerant strongly Menger edge connected for $n\geq 4$. Our results are optimal with respect to the maximum number of faulty edges. Furthermore, the results also show that the line graphs of $n$-dimensional hypercube-like networks can tolerate more faulty edges than $n$-dimensional hypercube-like networks \cite{Li5} to maintain strongly Menger edge connected. Compared with other data center networks, the data center networks that constructed by the line graphs of hypercube-like networks are high network capacity and good fault-tolerant.

\section*{Acknowledgement}

This work was supported by the National Natural Science Foundation of China [Grant number, 11971406, 12171402].

\begin{appendices}
	
\section*{Appendix A}
	
For $n=4$, $L_4\in{\cal L}_4$ and $S\subset E(L_4)$, if $|S|\leq 11$, then $L_4-S$ has a connected component with at least $30$ vertices.
	
\begin{proof}
	
Assume $L_4=L(Q^1_{3}\oplus_fQ^2_{3})$. Let $E_1=E(L_{3}^1),E_2=E(L_{3}^2),E_f=E(F_4,L_{3}^1\cup L_{3}^2)$ and $S_1=S\cap E_1,S_2=S\cap E_2,S_f=S\cap E_f$. We can see that $E(L_4)=E_1\cup E_2\cup E_f$ and $S=S_1\cup S_2\cup S_f$ are partitions of $E(L_4)$ and $S$, respectively. Without loss of generality, we assume $|S_1|\leq |S_2|$ and hence, $|S_1|\leq \lfloor|S|/2\rfloor=5$. We consider four cases.

{\bf Case 1.} $L_3^1-S_1$ and $L_3^2-S_2$ are both connected.
		
By Proposition \ref{prop}, every vertex in $F_4$ is adjacent to $3$ vertices in $L_3^1$ and $3$ vertices in $L_3^2$. Since $|S|\leq 11$, except possible two vertices $u_0, u_1\in F_4$, every vertex in $F_4\setminus \{u_0,u_1\}$ is adjacent to both a vertex in $L_3^1$ and a vertex in $L_3^2$, this means that $L_4-S-\{u_0,u_1\}$ is connected, which has exactly $30$ vertices.

{\bf Case 2.} $L_3^1-S_1$ is disconnected and $L_3^2-S_2$ is connected.

Since $|S|\leq 11$, we have $|S_1|\leq 5$. On the other hand, by Lemma \ref{lem 2.4.}, $\lambda(L_3^1)=4$, therefore $|S_1|\geq 4$, meaning that $4\leq |S_1|\leq 5$ and $L_3^1-S_1$ has a component with an isolated vertex $v_0$. Further, since $|S_2|\geq |S_1|$, we have $|S_2|\geq 4$ and thus $|S_f|\leq |S|-|S_1|-|S_2|=11-8=3$, and by Proposition \ref{prop}, except possible one vertex $u_0\in F_4$, every vertex in $F_4\setminus \{u_0\}$ is adjacent to both a vertex in $L_3^1$ and a vertex in $L_3^2$, this means that $V(L_4)\setminus \{v_0,u_0\}$ is connected, which has exactly $30$ vertices.

{\bf Case 3.} $L_3^1-S_1$ is connected and $L_3^2-S_2$ is disconnected.
		
Since $\lambda(L_3^2)=4$ and $|S|\leq 11$.
		
{\bf Case 3.1.} $4\leq |S_2|\leq 5$
		
In this case, $L_3^2-S_2$ has a connected component $H_2$ with $11$ vertices and the other component has an isolated vertex, say $v_0\in L_3^2$. Since $|S_2|\geq 4$, we have  $|S_f|\leq |S|-|S_2|\leq 7$, this implies that, except two possible vertices $u_0,u_1\in F_4$, every vertex in $F_4\setminus \{u_0,u_1\}$ is adjacent to both a vertex in $L_3^1$ and a vertex in $L_3^2$. Further, at least one vertex in $F_4\setminus \{u_0,u_1\}$ is adjacent to a vertex in $H_2$. For otherwise, we would have $|S|\geq 2(|F_4|-2)=12>11$, because $|F_4|=2^3=8$ and every vertex in $F_4$ is adjacent to $3$ vertices in $L_3$ and hence adjacent to $2$ vertices in $H_2$. This is a contradiction. As a result, the subgraph of $L_4-S$ induced by $V(L_3^1)\cup V(H_2)\cup (F_4\setminus \{u_0,u_1\})$ is connected. Therefore, if $v_0$ is adjacent to a vertex in $F_4\setminus \{u_0,u_1\}$, then we obtain a connected component with vertex set $V(L_4)\setminus \{u_0,u_1\}$.

We now assume that $v_0$ is not adjacent to any vertex in $F_4\setminus \{u_0,u_1\}$. In this case, $v_0$ is adjacent to at least one vertex in $\{u_0,u_1\}$ or is an isolated vertex in $L_4-S$. This implies that the number of edges in $S$ incident with $v_0$ is at least $4$ and, conversely, $S$ has at most $7$ edges that are not incident with $v_0$. Notice that $u_0$ and $u_1$ has degree $12$ in $L_4$, meaning that $u_0$ and $u_1$ has at least $5$ in $L_4-S$. Therefore, $u_0,u_1$ has at least one vertex, say $u_0$, in $L_3^1\cup H_2$. Hence, the subgraph induced by $V(L_3^1)\cup V(H_2)\cup (F_4\setminus \{u_1\})$ is connected, which has exactly $30$ vertices.
		
{\bf Case 3.2.} $6\leq |S_2|\leq 11$
		
In this case, we have $ |S_f|\leq |S|-|S_2|\leq 5 $, $L_3^2-S_2$ has a connected component $H_2$ with $10$ vertices, so by Proposition \ref{prop}, except possible one vertex, say $u_0\in F_4$, every vertex in $F_4\setminus \{u_0\}$ is adjacent to both a vertex in $L_3^1$ and a vertex in $L_3^2$. On the other hand, $u_0$ has degree $6$ in $L_4$ and hence, $u_0$ is adjacent to at least one vertex in $L_3^1\cup H_2$. Therefore, we can obtain a connected component with vertices $V(L_3^1)\cup V(H_2) \cup F_4$.

{\bf Case 4.} $L_3^1-S_1$ and $L_3^2-S_2$ are both disconnected.
		
Since $4\leq |S_1|\leq |S_2|$ and $|S_1|\leq 5<6$, by Lemma \ref{lem 3.2.}, $L_3^1$ has a component $H_1$ with $11$ vertices and, hence $V(L_3^1-H_1)$ has only one vertex, say $v_0$. Notice that $|S_2|\leq |S|-|S_1|\leq 7$, $L_3^2$ has a connected component $H_2$ with at least $10$ vertices and, hence $V(L_3^2-H_2)$ has at most two vertices, say $w_0$ and $w_1$. Thus, $|S_f|=|S|-|S_1|-|S_2|\leq 11-8=3<6$. By Proposition \ref{prop}, except possible one vertex, say $u_0\in F_4$, every vertex in $F_4\setminus \{u_0\}$ is adjacent to both a vertex in $L_3^1$ and a vertex in $L_3^2$.

We now assume that none of $v_0$, $w_0$ and $w_1$ is adjacent to any vertex in $F_4\setminus\{u_0\}$. In this case, $v_0, w_0$ and $w_1$ are adjacent only to $u_0$ or are isolated vertices in $L_4-S$, respectively. This implies that the number of edges in $S$ incident with $v_0, w_0$ and $w_1$ is at least $14>11\geq |S|$. This is a contradiction. Therefore, at most two vertices of $\{v_0,w_0,w_1\}$, say $v_0$ and $w_0$, are not adjacent to vertices in $H_1 \cup H_2$. On the other hand, since $|S_f|\leq 3<6$ and by Proposition \ref{prop}, $u_0$ is adjacent to at least one vertices in $H_1 \cup H_2$. Hence, the subgraph induced by $V(L_4)\setminus\{v_0,w_0\}$ is connected, which has  $30$ vertices.
\end{proof}
	
\end{appendices}


\begin{thebibliography}{25}
	
	\bibitem{Abu-Libdeh} H. Abu-Libdeh, P. Costa, A. Rowstron, G. O’Shea, A. Donnelly. Symbiotic routing in future data centers. In Proc. ACM
	SIGCOMM, Aug.30-Sept.3, 2010, pp.51-62.
	
	\bibitem{Al-Fares} M. Al-Fares, A. Loukissas, and A. Vahdat, A scalable, commodity data center network architecture, ACM SIGCOMM Comput. Commun. Rev. 38 (4) (2008) 63-74.
	
	\bibitem{Bondy} J.A. Bondy and U.S.A. Murty, Graph theory with applications, Elsevier, New York (1976).
	
	\bibitem{Chartrand} G.S. Chartrand, M.J. Stewart, The connectivity of line-graphs, Math. Ann. 182 (1969) 170-174.
	
	\bibitem{Chen} Y. C. Chen, M. H. Chen, J.J.M. Tan, Maximally local connectivity and connected components of augmented cubes, Inform. Sci. 273 (2014) 387-392.
	
	\bibitem{Cheng2} Q. Cheng, P. S. Li, M. Xu, Conditional (edge-)fault-tolerant strong Menger (edge) connectivity of folded hypercubes, Theoret. Comput. Sci. 728 (2018) 1-8.
	
	\bibitem{Cull} P. Cull, S.M. Larson, The M$\ddot{o}$bius cubes, IEEE Trans. Comput. 44 (1995) 647-659.
	
	\bibitem{Efe} K. Efe, The crossed cube architecture for parallel computation, IEEE Trans. Parallel and Distrib. Syst. 3 (5) (1992) 513–524.
	
	\bibitem{Fan} J.X. Fan, L.Q. He, BC interconnection networks and their properties, Chin. J. Comput. 26 (1) (2003) 84-90.
	
	\bibitem{Guo} C. X. Guo, G. H. Lu, D. Li, H. T. Wu, X. Zhang, Y. F. Shi, C. Tian, Y. G. Zhang, S. W. Lu, BCube: A high performance, servercentric network architecture for modular data centers, In Proc. the ACM SIGCOMM Conf. Data Communication. (2009) 63-74.
	
	\bibitem{Kan} S.X. Kan, J.X. Fan, B.L. Cheng, X. Wang, The Communication Performance of BCDC Data Center Network, 2020 12th International Conference on Communication Software and Networks (ICCSN). 2020, 51-57.
	
	\bibitem{Li4} P.S. Li, M. Xu, Fault-tolerant strong Menger (edge) connectivity and $3$-extra edge-connectivity of balanced hypercubes, Theoret. Comput. Sci. 707(2018) 56-68.
	
	\bibitem{Li5} P.S. Li, M. Xu, Edge-fault-tolerant strong Menger edge connectivity on the class of hypercube-like networks, Discrete Appl. Math. 259 (2019) 145-152.
	
	\bibitem{Li6} X.J. Li, J.M. Xu, Edge-fault-tolerance of hypercube-like networks, Inform. Process. Lett. 113 (19-21) (2013) 760-763.
	
	\bibitem{Lv1} M. Lv, S. Zhou, X. Sun, G. Lian, and J. Liu, Reliability evaluation of data center network DCell, Parallel Processing Letters. 28(04) (2018) 1850015.
	
	\bibitem{Lv2} M. Lv, B. Cheng, J. Fan, X. Wang, J. Zhou and J. Yu, The Conditional Reliability Evaluation of Data Center Network BCDC, The Computer Journal. 64 (9) (2020) 1451-1464.
	
	\bibitem{Menger} K. Menger, Zur allgemeinen kurventheorie, Fundamenta Mathematicae, 10 (1) (1927) 96-115.
	
	\bibitem{Ma} M. J. Ma, J. G. Yu, Note Edge-disjoint paths in faulty augmented cubes, Discrete Appl. Math. 294 (2021) 108-114.
	
	\bibitem{Oh} E. Oh, J. Chen, On strong Menger-connectivity of star graphs, Discrete Appl. Math. 129 (2-3)(2003) 499-511.
	
	\bibitem{Qiao} Y. Qiao, W. Yang, Edge disjoint paths in hypercubes and folded hypercubes with conditonal faults. Appl. Math. Comput. 294 (2017) 96-101.
	
	\bibitem{Vaidya} A.S. Vaidya, P.S.N. Rao, S.R. Shankar, A class of hypercube-like networks, in: Proc. of the 5th IEEE Symposium on Parallel and Distributed Processing, (1993) 800-803.	
	
	\bibitem{Wang} X. Wang, J. X. Fan, C.-K. Lin, and J. Y. Zhou, BCDC: A high-performance, server-centric data center network. J. Comput. Sci. Technol. 33 (2018) 400-416.
	
	\bibitem{Yang} X. Yang, D.J. Evans, G.M. Megson, The locally twisted cubes, Int. J. Comput. Math. 82 (4) (2005) 401-413.
	
	\bibitem{Zhang} Fuji Zhang, Guoning Lin, The line graph of hypercube. Journal of Xinjiang University: Natural Science Edition, 10 (4) (1993) 1-4.
	
	\bibitem{Zhou} J. X. Zhou. On $g$-extra connectivity of hypercube-like networks. J. Comput. Syst. Sci. 88 (2017) 208-219.	
	
\end{thebibliography}
\end{document}